\documentclass[11pt]{article}
\usepackage{amsmath, amscd, amssymb, latexsym, epsfig, color, amsthm}
\usepackage{tikz, enumerate}
\usetikzlibrary{calc}
\usepackage{pstricks,graphicx,subfig}
\numberwithin{equation}{section}

\newtheorem{theorem}{Theorem}[section]
\newtheorem{proposition}[theorem]{Proposition}

\newtheorem{conjecture}[theorem]{Conjecture}
\newtheorem{lemma}[theorem]{Lemma}

\theoremstyle{definition}

\newtheorem{remark}[theorem]{Remark}

\DeclareMathOperator\lk{\mathrm{lk}}
\DeclareMathOperator\st{\mathrm{st}}

\newcommand{\cupdot}{\mathbin{\mathaccent\cdot\cup}}

\newcommand{\field}{{\bf k}}

\newcommand{\Sp}{\mathbb{S}}

\title{The upper bound theorem for flag homology 5-manifolds}
\author{Hailun Zheng\\
	\small Department of Mathematics \\[-0.8ex]
	\small University of Michigan\\[-0.8ex]
	\small Ann Arbor, MI 48109, USA\\[-0.8ex]
	\small \texttt{hailunz@umich.edu}
}
\begin{document}
	\maketitle
	\begin{abstract}
		We prove that among all flag homology $5$-manifolds with $n$ vertices, the join of $3$ circles of as equal length as possible is the unique maximizer of all the face numbers. The same upper bounds on the face numbers hold for $5$-dimensional flag Eulerian normal pseudomanifolds.
	\end{abstract}
\section{Introduction}
A simplicial complex $\Delta$ is called \emph{flag} if all of its minimal non-faces are two-element sets. Flag complexes form a fascinating family of simplicial complexes. By the definition, the clique complex of any graph is a flag complex. Important classes of flag complexes include the barycentric subdivision of regular CW complexes and Coxeter complexes. 

We denote by $f_i$ the number of $i$-dimensional faces in a $(d-1)$-dimensional simplicial complex $\Delta$. The vector $(f_{-1}, f_0, \ldots, f_{d-1})$ is called the $f$\emph{-vector} of $\Delta$. Face enumeration is an active research field in geometric combinatorics, yet many problems on the face numbers of flag complexes remain open. The Kruskal-Katona theorem \cite{Ka, Kr} characterizes the $f$-vectors of simplicial complexes. For flag complexes, tight bounds on face numbers are not so well-understood, see, for example, \cite{F}. In particular, in the class of polytopes (or simplicial spheres) of a fixed dimension and with a fixed number of vertices, the celebrated upper bound theorem \cite{M}, \cite{St} states that neighborly polytopes (or neighborly spheres) simultaneously maximize all face numbers. Since neighborly $(d-1)$-spheres are not flag $(d-1)$-dimensional complexes, it is natural to ask whether there is a sharp upper bound conjecture for flag simplicial spheres.

Denote by $J_m(n)$ the simplicial $(2m-1)$-sphere with $n$ vertices obtained as the join of $m$ copies of the circle, each one a cycle with either $\lfloor \frac{n}{m} \rfloor$ or $\lceil \frac{n}{m}\rceil$ vertices. Also denote by $J_m^*(n)$ the suspension of $J_m(n-2)$. Several upper bound conjectures on the face numbers of flag homology spheres was proposed in recent years, as summarized in the following.
\begin{conjecture}[Nevo-Petersen, {\cite[Conjecture 6.3]{NP}}]\label{conj: Nevo-Petersen}
	The $\gamma$-vector of any flag homology sphere satisfies Frankl-F{\"u}redi-Kalai inequalities (see \cite{FFK}). 
\end{conjecture}

In particular, if this conjecture holds, then the face numbers of flag $(2m-1)$-spheres with $n$ vectices are simultaneously maximized by the face numbers of $J_m(n)$.

\begin{conjecture}[Nevo-Lutz, {\cite[Conjecture 6.3]{LN}}]\label{conj: Nevo-Lutz}
	In the class of flag homology $(2m-1)$-spheres with $n$ vectices, $J_m(n)$ is the unique maximizer of all the face numbers.
\end{conjecture}
Conjecture \ref{conj: Nevo-Petersen} also implies the following upper bound conjecture for even-dimensional flag spheres.
\begin{conjecture}[Adamaszek-Hladk{\'y}, {\cite[Conjecture 18]{AH}}]\label{conj: A-H}
     The face numbers of any flag homology $2m$-sphere with $n$ vertices are simultaneously maximized by those of $J_m^*(n)$.
\end{conjecture}
We remark that in contrast to the odd-dimensional conjecture, the conjectured maximizer of each face number for even-dimensional flag spheres is not unique, see the discussion in \cite{Z}.

Conjectures \ref{conj: Nevo-Petersen} and \ref{conj: Nevo-Lutz} have been confirmed in several cases. Adamaszek and Hladk{\'y} \cite{AH} proved an asymptotic version of the conjectures in the class of flag simplicial $(2m-1)$-manifolds on $n$ vertices. Indeed they showed that not only the $f$-numbers, but also $h$-numbers, $g$-numbers and $\gamma$-numbers are maximized by $J_m(n)$ as long as the number of vertices is large enough. More recently, Zheng \cite{Z} proved Conjectures \ref{conj: Nevo-Petersen} and \ref{conj: Nevo-Lutz} for all 3-dimensional flag simplicial manifolds. The same upper bound also holds for all 3-dimensional flag Eulerian complexes.

In this manuscript, we extend the technique used in \cite{Z} (namely, an application of the inclusion-exclusion principle) to prove that the Nevo-Petersen conjecture and the Nevo-Lutz conjecture also hold for flag homology 5-manifolds, see Theorem \ref{thm: main result}. Our proof reveals a strong relation between the odd-dimensional conjecture and even-dimensional conjecture: we show that if Conjecture \ref{conj: A-H} holds for the facet number of flag homology $2m$-spheres, then both Conjecture \ref{conj: Nevo-Petersen} and \ref{conj: Nevo-Lutz} holds for the facet number of flag homology $(2m+1)$-manifolds. The same argument applies to prove that Conjecture \ref{conj: Nevo-Petersen} continues to hold in the class of 5-dimensional flag Eulerian normal pseudomanifolds.

In \cite{G} Gal not only disproved the real rootedness conjecture but also proposed the following  conjecture. 
\begin{conjecture}[Gal, {\cite[Conjecture 3.2.2]{G}}]\label{conj: Gal}
	If $\Delta$ is a flag simplicial 3-sphere with $n$ vertices and $f_1(\Delta)> \frac{n^2}{4}+\frac{n}{2}+4$, then $\Delta$ is the join of two cycles. 
\end{conjecture}
Adamaszek and Hladk{\'y} \cite{AH1} verified this conjecture in the class of flag 3-manifolds with sufficiently many vertices. In this manuscript we show that for $m=2,3$, there exists a constant $b(m)>0$ such that if $\Delta$ is a flag homology $(2m-1)$-manifold with $n$ vertices and $f_{2m-1}(\Delta)\geq f_{2m-1}(J_m(n))-b(m)n^{m-1}$, then $\Delta$ is indeed the join of two cycles.

The structure of this manuscript is as follows. Section 2 contains basic definitions and properties related to simplicial complexes and flag complexes. In Section 3, we discuss the upper bound results on flag homology 5-manifolds and 5-dimensional Eulerian complexes, as well as a generalization of Gal's conjecture.
\section{Preliminaries}
	A \emph{simplicial complex} $\Delta$ on a vertex set $V=V(\Delta)$ is a collection of subsets
	$\sigma\subseteq V$, called faces, that is closed under inclusion. For $\sigma\in \Delta$, let $\dim\sigma:=|\sigma|-1$ and define the \emph{dimension} of $\Delta$, $\dim \Delta$, as the maximal dimension of its faces. A \emph{facet} in $\Delta$ is a maximal face under inclusion, and we say that $\Delta$ is \emph{pure} if all of its facets have the same dimension. We denote by $\cupdot$ the disjoint union of sets.
	
	For a $(d-1)$-dimensional complex $\Delta$, we let $f_i = f_i(\Delta)$ be the number of $i$-dimensional faces of $\Delta$ for $-1\leq i\leq d-1$. The vector $(f_{-1}=1, f_0, \ldots, f_{d-1})$ is called the $f$\emph{-vector} of $\Delta$. If $\Delta$ is a simplicial complex and $\sigma$ is a face of $\Delta$, the \emph{link} of $\sigma$ in $\Delta$ is $\lk(\sigma,\Delta):=\{\tau-\sigma\in \Delta: \sigma\subseteq \tau\in \Delta\}$, and the \emph{star} of $\sigma$ in $\Delta$ is $\st(\sigma,\Delta):=\{\tau\in \Delta: \sigma\cup \tau\in \Delta\}$. When the context is clear, we will abbreviate the notation and write them as $\lk(\sigma)$ and $\st(\sigma)$ respectively. The \emph{deletion} of a vertex set $W$ from $\Delta$ is $\Delta\backslash W:=\{\sigma\in\Delta:\sigma\cap W=\emptyset\}$. The \emph{restriction} of $\Delta$ to a vertex set $W$ is defined as $\Delta[W]:=\{\sigma\in \Delta:\sigma\subseteq W\}$. If $\Delta$ and $\Gamma$ are two simplicial complexes on disjoint vertex sets, then the \emph{join} of $\Delta$ and $\Gamma$, denoted as $\Delta *\Gamma$, is the simplicial complex on vertex set $V(\Delta)\cupdot V(\Gamma)$ whose faces are $\{\sigma\cup\tau:\sigma\in\Delta, \tau\in\Gamma\}$.
	
	Denote by ${H}_*(\Delta; \field)$ the \emph{homology} of $\Delta$ with coefficients in $\field$ and by $\beta_i(\Delta;\field)=\dim_\field {H}_i(\Delta; \field)$ the \emph{Betti numbers} of $\Delta$. A simplicial complex $\Delta$ is a $\field$-\emph{homology manifold} if the homology group ${H}_*(\lk(\sigma), \field)\cong {H}_*(\Sp^{d-1-|\sigma|}, \field)$ for all nonempty face $\sigma\in\Delta$. A $\field$-homology sphere is a $\field$-homology manifold that has the $\field$-homology of a sphere. A $(d-1)$-dimensional simplicial complex $\Delta$ is called a \emph{$(d-1)$-pseudomanifold} if it is pure and every $(d-2)$-face (called \emph{ridge}) of $\Delta$ is contained in exactly two facets. A $(d-1)$-pseudomanifold $\Delta$ is called a \emph{normal $(d-1)$-pseudomanifold} if it is connected, and the link of each face of dimension at most $d-3$ is also connected. Every normal 2-pseudomanifold is a homology 2-manifold. However, for $d>3$, the class of normal $(d-1)$-pseudomanifolds is much larger than the class of homology $(d-1)$-manifolds. Also the family of pseudomanifolds (normal pseudomanifolds, resp.) is closed under taking links.
	
	We let $\chi(\Delta)=\sum_{i=0}^{d-1}(-1)^i\beta_i(\Delta;\field)$ be the \emph{Euler characteristic} of $\Delta$. We say $\Delta$ is \emph{Eulerian} if $\Delta$ is pure and $\chi(\lk(\sigma)) = (-1)^{\dim \lk(\sigma)}+1$ for every $\sigma\in\Delta$ including $\sigma = \emptyset$. Eulerian complexes are pseudomanifolds and it follows from the
	Poincar\'e duality theorem that all odd-dimensional orientable homology manifolds are Eulerian. Recall that the \emph{$h$-vector} of $\Delta$, $(h_0,h_1, \dots, h_d)$, is defined by the relation $\sum_{i=0}^d h_i t^{d-i}=\sum_{i=0}^d f_{i-1}(t-1)^{d-i}$. The following lemma is the well-known Dehn-Sommerville relations for Eulerian complexes. 
	\begin{lemma}\label{lm: dehn-sommerville}
		Let $\Delta$ be a $(d-1)$-dimensional Eulerian complex. Then $h_i(\Delta)=h_{d-i}(\Delta)$ for all $0\leq i\leq d$. In particular, if $d=6$, then the $f$-vector of $\Delta$ can be written as
		\[(f_0, f_1, f_2=f_5+2f_1-2f_0, f_3=3f_5+f_1-f_0, f_4=3f_5, f_5).\]
	\end{lemma}
	
	A simplicial complex $\Delta$ is \emph{flag} if all minimal non-faces of $\Delta$, also called missing faces, have cardinality two; equivalently, $\Delta$ is the clique complex of its graph. A crucial property of flag complexes is described in the following lemma.
	\begin{lemma}\label{lm: flag prop}
		Let $\Delta$ be a flag complex on vertex set $V$.
		\begin{enumerate}
		\item If $W\subseteq V(\Delta)$, then $\Delta[W]$ is also flag. 
		\item If $\sigma$ is a face in $\Delta$, then $\lk(\sigma)=\Delta[V(\lk(\sigma))]$. In particular, all links in a flag complex are also flag.
		\item Any edge $\{v,v'\}$ in $\Delta$ satisfies the link condition $\lk(v)\cap \lk(v')=\lk(\{v,v'\})$. More generally, any face $\sigma=\sigma_1\cup \sigma_2$ in $\Delta$ ($\sigma_1$ and $\sigma_2$ are not necessarily disjoint) satisfies $\lk(\sigma)=\lk(\sigma_1)\cap \lk(\sigma_2)$.
		\end{enumerate}
	\end{lemma}
	\begin{proof}
		The first two claims are essentially Lemma 5.2 in \cite{NP}. We prove the last claim. If $\sigma=\sigma_1\cup \sigma_2$, then $\lk(\sigma)\subseteq\lk(\sigma_1)\cap \lk(\sigma_2)$ always holds. Conversely, let $\tau \in \lk(\sigma_1)\cap \lk(\sigma_2)$. Every vertex of $\tau$ is connected to every vertex in $\sigma_1$ and $\sigma_2$. Also since $\sigma=\sigma_1\cup\sigma_2$ is a face of $\Delta$, the complete graph on $V(\tau)\cupdot V(\sigma)$ is a subgraph of $G(\Delta)$. So $\Delta$ is flag implies that $\sigma\cup \tau$ forms a face of $\Delta$, i.e., $\tau\in \lk(\sigma)$. Hence $\lk(\sigma_1)\cap \lk(\sigma_2)\subseteq \lk(\sigma)$.
	\end{proof}
	
 Finally, we describe a class of flag odd-dimensional simplicial spheres. For $m\geq 2$ and $n\geq 4m$, define $J_m(n)$ as the flag simplicial $(2m-1)$-sphere on $n$ vertices obtained as the join of $m$ circles, each one of length either $\lfloor \frac{n}{m} \rfloor$ or $\lceil \frac{n}{m}\rceil$. In particular, $J_m(4m)$ is the boundary complex of the $2m$-dimensional cross-polytope. If we let $J_m^*(n+2)$ be the suspension of $J_m(n)$ for $n\geq 4m$, then $J_m^*(n+2)$ is a flag $2m$-dimensional sphere. Write $n=qm+k$ where $q, k$ are integers and $0\leq k< m$. We have that $$f_1(J_m(n))=\frac{1}{2}\Big( (n-q+1)k(q+1)+(n-q+2)(m-k)q\Big) =n+\frac{m-1}{2m}n^2+\frac{k^2-km}{2m}\leq  n+\frac{m-1}{2m}n^2$$ $$\text{and}\;\;f_{2m}(J_m^*(n+2))=2f_{2m-1}(J_m(n))=2(q+1)^kq^{m-k}\leq 2(\frac{n}{m})^m.$$ Both of the above inequalities hold as equalities if and only if $n$ is a multiple of $m$.
The following lemma \cite[Lemma 5.6]{Z} gives a necessary and sufficient condition for a flag normal pseudomanifold to be the join of two of its face links.
\begin{lemma}\label{lm: join of face links}
	Let $\Delta$ be a flag normal $(d-1)$-pseudomanifold and let $\sigma=\sigma_1\cupdot \sigma_2$ be a facet of $\Delta$. Then $V(\lk(\sigma_1))\cupdot V(\lk(\sigma_2))=V(\Delta)$ if and only if $\Delta=\lk(\sigma_1)*\lk(\sigma_2)$.
\end{lemma}

\section{Proof of the main theorem}
The goal of this section is to prove the upper bound conjecture for flag homology 5-manifolds, see Theorem \ref{thm: main result}. Let $\Delta$ be a flag $(2m-1)$-dimensional pseudomanifold. Our strategy is to use the inclusion-exclusion principle to give an upper bound on $\sum_{v\in\sigma} f_0(\lk(v))$, where $\sigma$ is a facet of $\Delta$, and then use this result to obtain an upper bound of $f_{2m-1}(\Delta)$. We begin with a generalization of Lemmas 5.2 and 5.3 in \cite{Z}.

\begin{lemma} \label{lemma: sum of codim 2 face links}
	Let $\Delta$ be a flag $(d-1)$-pseudomanifold with $n$ vertices and let $\sigma$ be a facet of $\Delta$. Then $\sum_{\tau\subset\sigma, |\tau|=d-2} f_0(\lk(\tau))= |\cup_{\tau\subset\sigma, |\tau|=d-2}V(\lk(\tau))|+2d(d-2)$.
\end{lemma}
\begin{proof}
	Let $\ell=\binom{d}{2}$ and let $V_1=V(\lk(\sigma_1)),\dots, V_\ell=V(\lk(\sigma_\ell))$ denote the vertex sets of the links of $(d-3)$-faces of $\sigma$. Note that $\sigma_1\cup \sigma_2$ is either a $(d-1)$-face or a $(d-2)$-face. By Lemma \ref{lm: flag prop}, $V_i\cap V_j=V(\lk(\sigma_1\cup \sigma_2))$ is either the empty set, or the link of a ridge, which is the disjoint union of two vertices (since $\Delta$ is a pseudomanifold). Hence by the inclusion-exclusion principle,
	\begin{align*}
	\begin{split}
	\sum_{i=1}^{\ell} |V_i|&=|\cup_{i=1}^{\ell} V_i|+\sum_{1\leq i<j\leq \ell} |V_i\cap V_j|-\sum_{1\leq i<j<k\leq \ell} |V_i\cap V_j\cap V_k|+\dots\\
	&=|\cup_{i=1}^{\ell} V_i|+2\cdot\Big(\#\{(i,j):|\sigma_i\cup\sigma_j|=d-1\}-\#\{(i,j,k):|\sigma_i\cup\sigma_j\cup \sigma_k|=d-1\}+\dots\Big)\\
	&= |\cup_{i=1}^{\ell} V_i|+2\cdot \binom{d}{d-1}\cdot\left(\binom{d-1}{2}-\binom{d-1}{3}+\dots +(-1)^{d-1}\binom{d-1}{d-1}\right)\\
	&=|\cup_{i=1}^{\ell} V_i|+2d(d-2).
	\end{split}
	\end{align*}
\end{proof}

In the following, let $\Delta$ be a flag $(2m-1)$-pseudomanifold on $n$ vertices, $\sigma$ a facet of $\Delta$, and $\tau$ a $(k-1)$-face in $\sigma$, where $k\leq 2m-2$. We denote by $W_\tau$ the set of vertices in $\Delta$ connected to $\tau$ but not connected to any vertices in $\sigma\backslash\tau$. Define $a_k:=\sum_{\tau\subset\sigma, |\tau|=k}f_0(\lk(\tau))$ and $b_k:=\sum_{\tau\subset\sigma, |\tau|=k}f_0(W_\tau)$. To give an explicit formula of $a_k$ in terms of $a_{2m-2}$ and $b_i$'s, we need the following identity.
\begin{lemma}\label{lm: identity}
	$\sum_{i=k+1}^{2m-2}(-1)^{i-k+1}\binom{i}{k}\binom{2m-2}{i}=\binom{2m-2}{k}$.
\end{lemma}
\begin{proof}
	Note that \begin{align*}
		\begin{split}
		\sum_{i=k+1}^{2m-2}(-1)^{i-k+1}\binom{i}{k}\binom{2m-2}{i}&=\sum_{i=k+1}^{2m-2}(-1)^{i-k+1}\binom{2m-2}{k}\binom{2m-2-k}{i-k}\\
		&=\binom{2m-2}{k}\sum_{j=1}^{2m-2-k}(-1)^{j-1}\binom{2m-2-k}{j}.
		\end{split}
	\end{align*}Substituting $x=1$ in both sides of the identity $(1-x)^{2m-2-k}=\sum_{j=0}^{2m-2-k}(-1)^{j}\binom{2m-2-k}{j}x^j$ yields that $\sum_{j=1}^{2m-2-k}(-1)^{j}\binom{2m-2-k}{j}=-1$. Hence the left hand side of the above equation equals $\binom{2m-2}{k}$. 
\end{proof}
\begin{lemma}\label{lm: a_k, b_k}
 	Let $a_k$, $b_k$ be defined as above. For $1< k\leq 2m-2$, \begin{equation}\label{formula: a_k}
 	a_k=\binom{2m-2}{k}a_{2m-2}-4m(2m-2-k)\binom{2m-1}{k}+\sum_{i=k}^{2m-3}\binom{i}{k}b_i.
 	\end{equation}
\end{lemma}
\begin{proof}
	 Since $\Delta$ is flag, the link $\lk(\tau)$ is the induced subcomplex of $\Delta$ on $V(\lk(\tau))$. It follows that $|\cup_{v\in \sigma\backslash\tau}V(\lk(\tau\cup v))|=f_0(\lk(\tau))-f_0(W_\tau)$. By the inclusion-exclusion principle,
	 \begin{equation*}
	 f_0(\lk(\tau))-f_0(W_\tau)=\sum_{v\subset \sigma\backslash\tau}
	 f_0(\lk(\tau\cup v))-\sum_{e\subset \sigma\backslash\tau}f_0(\lk(\tau\cup e))+\dots +(-1)^{k}\sum_{\delta\subset \sigma\backslash\tau, |\delta|=2m-k-1}f_0(\lk(\tau\cup\delta)).
	 \end{equation*}
	Take the sum over all $(k-1)$-faces in $\sigma$ to obtain that
	 \begin{equation}\label{eq:1}
	 \begin{split}
	 a_k-b_k&=\binom{k+1}{k}a_{k+1}-\binom{k+2}{k}a_{k+2}+\cdots+(-1)^{k}\binom{2m-1}{k}a_{2m-1}.
	 \end{split}
	 \end{equation}
	 First note that the formula (\ref{formula: a_k}) of $a_k$ holds for $k=2m-2$. Assume that the formula (\ref{formula: a_k}) holds for all $a_i$ with $2m-2\geq i>k$. Now we will express $a_k$ as a linear combination of $a_{2m-2}, b_k, b_{k+1}, \dots, b_{2m-3}$ and 1. By expanding each $a_{k+1}, \dots, a_{2m-2}$ in (\ref{eq:1}) using (\ref{formula: a_k}),  the coefficient of $a_{2m-2}$ in $a_k$ is \[\mathrm{coeff}_{a_k}a_{2m-2}=\sum_{i=k+1}^{2m-2}(-1)^{i-k+1}\binom{i}{k}\binom{2m-2}{i}=\binom{2m-2}{k}\]
	 by Lemma \ref{lm: identity}. Next we consider the constant term in the expression of $a_{2m-2}$. Since $\Delta$ is a pseudomanifold, every ridge is contained in exactly two facets, and hence $a_{2m-1}=\binom{2m}{2m-1}\cdot 2=4m$. We use the fact that \[(2m-2-k)\binom{2m-1}{k}=\binom{2m-2}{k}(2m-1)-\binom{2m-1}{k}\] along with our inductive hypothesis and Lemma \ref{lm: identity} to obtain that
	 \begin{align*}
	 \begin{split}
	 \mathrm{coeff}_{a_k}1&=4m\cdot\sum_{i=k+1}^{2m-2}(-1)^{i-k+1}\binom{i}{k}\left(\binom{2m-2}{i}\cdot(2m-1)-\binom{2m-1}{i}\right)+(-1)^{2m-k}\binom{2m-1}{k}\cdot 4m\\
	 &=4m\cdot \left(\binom{2m-2}{k}\cdot(2m-1)-\binom{2m-1}{k}\right)\\
	 &=4m(2m-2-k)\binom{2m-1}{k}.
	 \end{split}
	 \end{align*}
	 Finally, denote by $c_k$ the remaining terms in $a_k$, which consists of a linear combination of $b_i$. By induction,
	 \begin{align*}
	 \begin{split}
	 c_k&=b_k+\sum_{i=k+1}^{2m-3} (-1)^{i-k+1}\binom{i}{k}\sum_{j=i}^{2m-3}\binom{j}{i}b_j\\
	 &=b_k+\sum_{j=k+1}^{2m-3}\binom{j}{k}b_j\sum_{i=k+1}^{j}(-1)^{i-k+1}\binom{j-k}{i-k}\\
	 &=b_k+\sum_{j=k+1}^{2m-3}\binom{j}{k}b_j\\
	 &=\sum_{j=k}^{2m-3}\binom{j}{k}b_j.
	 \end{split}
	 \end{align*} 
	 This finishes proving the claim.
\end{proof}

We are now ready to estimate $\sum_{v\in\sigma} f_0(\lk(v))$ ($\sigma$ is a facet) using the above lemma.
\begin{lemma}\label{lemma: sum of vertex links}
	Let $m\geq 2$ and let $\Delta$ be a flag $(2m-1)$-pseudomanifold with $n$ vertices. Also let $\sigma$ be a facet of $\Delta$. Then $\sum_{v\in\sigma} f_0(\lk(v))\leq 2(m-1)n +4m$. Furthermore if equality holds, then 
	\begin{itemize}
		\item $\cup_{v\in \sigma\backslash\tau} V(\lk(v\cup\tau))=V(\lk(\tau))$ for any face $\tau\subset\sigma$ with $|\sigma\backslash\tau|\geq 3$;
		\item $\cup_{\delta\subset \sigma, |\delta|=2m-2}V(\lk(\delta))=V(\Delta)$.
	\end{itemize}
\end{lemma}
\begin{proof}
    By the inclusion-exclusion principle,
    
	\begin{align*}
	\begin{split}
	a_1&=|\cup_{v\in \sigma}V(\lk(v))|+a_2-a_3+\cdots -a_{2m-1}\\
	&\stackrel{(*)}{=}|\cup_{v\in \sigma}V(\lk(v))|+\Big(a_{2m-2}-4m(2m-1)\Big)\cdot\sum_{i=2}^{2m-2}(-1)^i\binom{2m-2}{i}\\
	&\quad +4m\sum_{i=2}^{2m-2}(-1)^i\binom{2m-1}{i}+\sum_{k=2}^{2m-3}(-1)^k\sum_{i=k}^{2m-3}\binom{i}{k}b_i\\
	&=|\cup_{v\in \sigma}V(\lk(v))| +\Big(a_{2m-2}-4m(2m-1)\Big)(2m-3)+4m\big(2m-2\big)+\sum_{i=2}^{2m-3}\left(\sum_{k=2}^i (-1)^k\binom{i}{k}\right)b_i\\
	&\stackrel{(**)}{=}|\cup_{v\in \sigma}V(\lk(v))|+(2m-3)\Big(\left|\cup_{\delta\subset \sigma, |\delta|=2m-2}V(\lk(\delta))\right|+8m(m-1)\Big)\\
	&\quad-4m(4m^2-10m+5)+\sum_{i=2}^{2m-3}(i-1)b_i\\
	&=|\cup_{v\in \sigma}V(\lk(v))|+4m+(2m-3)\left|\cup_{\delta\subset \sigma, |\delta|=2m-2}V(\lk(\delta))\right|+\sum_{i=2}^{2m-3}(i-1)b_i.
	\end{split}	
	\end{align*}
	Here in (*) we use Lemma \ref{lm: a_k, b_k} and in (**) we use Lemma \ref{lemma: sum of codim 2 face links}. Note that all $W_\tau$ are disjoint vertex sets for different faces $\tau\subset\sigma$. Hence \[\left|\cup_{\delta\subset \sigma, |\delta|=2m-2}V(\lk(\delta))\right|+ \sum_{i=2}^{2m-3} b_i\leq n.\] So we conclude that
	\[a_1\leq n+4m+ (2m-3)\Big(n-\sum_{i=2}^{2m-3}b_i\Big)+\sum_{i=2}^{2m-3}(i-1)b_i\leq (2m-2)n+4m.\]
	When equality holds, we must have that $\cup_{v\in\sigma} V(\lk(v))=\cup_{\delta\subset \sigma, |\delta|=2m-2}V(\lk(\delta))=V(\Delta)$ and $b_i=0$ for $2\leq i\leq 2m-3$. The second condition ``$b_i=0$" is equivalent to  $\cup_{v\in \sigma\backslash\tau} V(\lk(v\cup\tau))=V(\lk(\tau))$ for any face $\tau\subset\sigma$ with $|\sigma\backslash\tau|\geq 3$.
\end{proof}

The next proposition gives the relation between even-dimensional and odd-dimensional upper bound conjecture for flag pseudomanifolds.
\begin{proposition}\label{theorem: upper bound inequality}
	Let $\Delta$ be a flag  $(2m-1)$-pseudomanifold with $n$ vertices. Assume that for each vertex $v\in\Delta$,
	\begin{equation}\label{eq: even-dimension conjecture}
		f_{2m-2}(\lk(v))\leq f_{2m-2}(J_{m-1}^*(f_0(\lk(v)))).
	\end{equation}
	Then $f_{i}(\Delta)\leq f_{i}(J_m(n))$ for $i=2m-2, 2m-1$.
\end{proposition}
\begin{proof}
	By Lemma \ref{lemma: sum of vertex links}, $\sum_{v\in\sigma}f_0(\lk(v))\leq 2(m-1)n+4m$ for any facet $\sigma$ of $\Delta$. Hence
	\begin{equation}\label{eq: upper bound of f_1}
		\big( 2(m-1)n+4m \big) f_{2m-1}(\Delta)\geq \sum_{\sigma\in\Delta,|\sigma|=2m}\sum_{v\in\sigma} f_0(\lk(v))=\sum_{v\in\Delta} f_0(\lk(v))f_{2m-2}(\lk(v)).
	\end{equation}
	The assumption on the vertex links is equivalent to
	\[f_0(\lk(v))\geq (m-1)\cdot\left( \frac{1}{2}f_{2m-2}(\lk(v))\right) ^{\frac{1}{m-1}}+2.\] Hence the RHS of (\ref{eq: upper bound of f_1}) is at least
	\[\sum_{v\in\Delta}2(m-1)\left( \frac{1}{2}f_{2m-2}(\lk(v))\right) ^{1+\frac{1}{m-1}}+2f_{2m-2}(\lk(v)).\]
    By the H\"older inequality, $$\sum_{v\in \Delta} \frac{1}{2}f_{2m-2}(\lk(v))\leq \left( \sum_{v\in\Delta} (\frac{1}{2}f_{2m-2}(\lk(v)))^{\frac{m}{m-1}}\right) ^{\frac{m-1}{m}}\left( \sum_{v\in\Delta}|1|^{m}\right) ^{\frac{1}{m}}.$$ Hence together with the fact that $\sum_{v\in\Delta}\frac{1}{2}f_{2m-2}(\lk(v))=mf_{2m-1}(\Delta)$, we conclude that the RHS of (\ref{eq: upper bound of f_1}) is also greater than 
	\[2n(m-1)\left( \frac{mf_{2m-1}(\Delta)}{n}\right) ^{1+\frac{1}{m-1}}+4mf_{2m-1}(\Delta).\]
	Finally plugging in back to (\ref{eq: upper bound of f_1}) and simplify, we obtained that $f_{2m-1}(\Delta)\leq (\frac{n}{m})^m$. Furthermore for $f_{2m-1}(\Delta)$ to attain the maximum, we must have that 1) $f_{2m-2}(\lk(v))=f_{2m-2}(J^*_{m-1}(f_0(\lk(v)))$ for all vertex $v\in \Delta$, and 2) the numbers of vertices among all vertex links are as even as possible. These two conditions lead to a better estimation $f_{2m-1}(\Delta)\leq f_{2m-1}(J_m(n))$. Since $f_{2m-2}(\Delta)=mf_{2m-1}(\Delta)$ for any pseudomanifold $\Delta$, the other inequality also holds.
\end{proof}


Now we are ready to prove our main upper bound results on flag homology 5-manifold.
\begin{theorem}\label{thm: main result}
	Let $\Delta$ be a flag 5-dimensional simplicial complex with $n$ vertices. If furthermore, $\Delta$ is either an Eulerian normal pseudomanifold or homology manifold, then $f_i(\Delta)\leq f_i(J_3(n))$ for all $0\leq i\leq 5$.
\end{theorem}
\begin{proof}
	Assume first that $\Delta$ is a Eulerian normal pseudomanifold. Since $\lk(v)$ is Eulerian, the link of every vertex has a palindromic $h$-polynomial. Hence Corollary 3.1.7 in \cite{G} applies and the $\gamma$-polynomial of $\lk(v)$ has at least one real root. This implies that $4\gamma_2(\lk(v))\leq \gamma_1(\lk(v))^2$ and hence inequality (\ref{eq: even-dimension conjecture}) holds. By Proposition \ref{theorem: upper bound inequality} we have that $f_5(\Delta)\leq f_5(J_3(n))$. The proof of the inequality $f_1(\Delta)\leq f_1(J_3(n))$ is exactly the same as the proof of Proposition 4.2 in \cite{Z}; we omit the details. Finally the other inequalities follow from the fact that the coefficients of $f_1(\Delta)$ and $f_5(\Delta)$ are nonnegative, in the expressions for the other $f_j(\Delta)$, see Lemma \ref{lm: dehn-sommerville}. The same proof also applies to any flag homology 5-manifold.
\end{proof}
\begin{remark}
	As indicated by the proof above, the inequality $f_5(\Delta)\leq f_5(J_3(n))$ also holds for any 5-dimensional flag Eulerian complex $\Delta$.
\end{remark}
Next we characterize the maximizer of face numbers of flag 5-manifolds. The following lemma gives a sufficient condition for a flag 2-sphere to be a suspension of a circle. For simplicity of notation, we write the edge link $\lk(\{u,v\})$ as $\lk(uv)$, and we say a facet $\sigma$ is adjacent to another facet $\sigma'$ if they share a ridge.
\begin{lemma}\label{lemma: flag 2-sphere property}
	Let $\Delta$ be a flag simplicial 2-manifold and $\sigma$ a facet of $\Delta$. If $\cup_{v\in\sigma}V(\lk(v))=\cup_{v\in\sigma'}V(\lk(v))=V(\Delta)$ for every facet $\sigma'$ that is adjacent to $\sigma$, then $\Delta$ is the suspension of a circle of length $\geq 4$.
	\end{lemma}
	\begin{proof}
		Let $\sigma=\{u,v,w\}$ be a facet of $\Delta$ and let $\{u,u'\}$, $\{v,v'\}$ and $\{w,w'\}$ be the vertices of $\lk(vw), \lk(uw), \lk(uv)$ respectively. Since $\Delta$ is flag, the intersection of the stars of two adjacent vertices in $\Delta$ must be the union of two 2-faces. Hence $\cup _{a\in\sigma}\st(a)$ is a triangulated 2-ball and its boundary $C$ contains the vertices $u',v',w'$. We write $C$ as the union of three paths $P_u=\lk(u)\cap C, P_v=\lk(v)\cap C$ and $P_w=\lk(w)\cap C$. Also denote by $P_u\backslash \partial P_u$ the path obtained from $P_u$ by removing its endpoints ($P_u\backslash \partial P_u$ can be a singleton or the emptyset). 
		
		If $C$ is a 3-cycle, then $\{u',v',w'\}\in\Delta$ and $\Delta$ is the octahedral sphere. Otherwise, at least one of $P_u, P_v, P_w$, say $P_u$, has more than two vertices. Since $\cup_{a\in\sigma}V(\lk(a))=\cup_{a\in\{u',v,w\}}V(\lk(a))=V(\Delta)$, it follows that $u'$ is connected to every vertex in $P_u\backslash\partial P_u$. Also every vertex link in $\Delta$ is the induced subcomplex on its vertices, so \[|\lk(v')\cap (P_u\cup P_w)|=|\lk(w')\cap (P_u\cup P_v)|=2.\] 
		Hence both $\lk(v')$ and $\lk(w')$ must be 4-cycles. However, from $\cup_{a\in \{v',u,w\}} V(\lk(a))=V(\Delta)$ it follows that $v'$ is connected to every vertex in $P_v\backslash\partial P_v$. So $P_v\backslash\partial P_v=\emptyset$ and $\lk(v)$ is a 4-cycle. Similarly $\lk(w)$ is also a 4-cycle. This implies that $\Delta$ is the suspension of the circle $\lk(u)$.  
		\end{proof}
		
\begin{lemma}\label{lemma: join of circles}
	Let $m\geq 2$ and $\Delta$ be a flag normal $(2m-1)$-pseudomanifold with $n$ vertices. Also let $\sigma$ be a facet of $\Delta$. If the link of every $(2m-4)$-face in $\sigma$ is the suspension of a circle and furthermore, $\cup_{\delta\subset \sigma, |\delta|=2m-2}V(\lk(\delta))=V(\Delta)$, then $\Delta$ is the join of $m$ circles.
\end{lemma}
\begin{proof}
	The proof is by induction on $m$. First let $m=2$. If $n=8$, then $\Delta=J_2(8)$. Otherwise let $n>8$. There exists an edge $\{v_1,v_2\}\subset \sigma$ whose link has more than four vertices. We let $\sigma=\{v_1,v_2\}\cup\{u_1,u_2\}$. Since $\lk(v_1), \lk(v_2)$ are suspensions of the circle $\lk(v_1v_2)$, the link $\lk(v_iu_j)$ must be a 4-cycle for any $i, j=1,2$. Then it follows from Lemma \ref{lemma: sum of codim 2 face links} that $f_0(\lk(v_1v_2))+f_0(\lk(v_3v_4))=n+16-4\cdot4=n$. Hence $V(\lk(v_1v_2))\sqcup V(\lk(v_3v_4))=V(\Delta)$ and $\Delta=\lk(v_1v_2)*\lk(v_3v_4)$ by Lemma \ref{lm: join of face links}.
	
	Next assume the claim holds for all flag normal $(2m-1)$-pseudomanifolds with $2\leq m<k$ and we let $m=k$. For any edge $e\subset\sigma$, since $\cup_{\delta\subset \sigma, |\delta|=2k-2}V(\lk(\delta))=V(\Delta)$, we obtain that
	\begin{equation}\label{eq: edge link cover}
		f_0(\lk(e, \Delta))=|\cup_{\delta\subset\sigma, |\delta|=2k-2}V(\lk(\delta, \Delta))\cap V(\lk(e,\Delta))|=|\cup_{\delta\subset\sigma\backslash e, |\delta|=2k-4}V(\lk(\delta, \lk(e))|.
	\end{equation}
	For a facet $\sigma\in \Delta$ that contains $e$, $\sigma\backslash e$ is a facet in $\lk(e)$. If $\tau$ is a $(2k-6)$-face in $\sigma\backslash e$, then $\lk(\tau, \lk(e))=\lk(\tau\cup e, \Delta)$ is the suspension of a circle. Hence by the identity (\ref{eq: edge link cover}) and the inductive hypothesis, the link $\lk(e)$ is the join of $k-1$ circles. 
	
	If $n=4k$, then $\Delta$ is the join of $k$ 4-cycles. (See \cite[Lemma 2.1.14]{G}; the argument continues to hold for flag normal pseudomanifolds.) Otherwise if $n>4k$, then choose a $(2k-3)$-face $\tau'\in\sigma$ whose link has the maximum number of vertices. Since $\Delta\neq J_k(4k)$ and $\cup_{\delta\subset \sigma, |\delta|=2k-2}V(\lk(\delta))=V(\Delta)$, we have $f_0(\lk(\tau'))>4$. Furthermore, the above argument shows that for $e_0=\sigma\backslash \tau$, $\lk(e_0)=C_1*C_2*\cdots*C_{k-1}$. Now in each $C_i$ there is an edge $e_i$ such that $\sigma=\cup_{i=0}^{k-1}e_i$ of $\Delta$. Let $\delta$ be a $(2k-3)$-face of $\Delta$. If $\delta=\sigma\backslash\{v_1,v_2\}$, where $v_1\in e_0$ and $v_2\in \cup_{i=1}^{k-1}e_i$, then $\delta\backslash e_0$ is a $(2k-4)$-face whose link is the suspension of a circle. Since $\lk(\tau\backslash e_0)$ is a circle of length $>4$, it follows that $\lk(\delta)$ must be a 4-cycle. Also when $e_0\subset \delta$, the link $\lk(\delta)=\lk(\delta\backslash e_0,\lk(e_0))$ is always a 4-cycle, except perhaps for $\delta=\sigma\backslash e_i$, $i=1,2,\cdots,k-1$. Hence by Lemma \ref{lemma: sum of codim 2 face links} and $\cup_{\delta\in \sigma, |\delta|=2m-2}V(\lk(\delta))=V(\Delta)$,
	\[\sum_{i=0}^{k-1}f_0(\lk(\sigma\backslash e_i))=n+8k(k-1)-\left(\binom{2k}{2}-k\right)\cdot 4=n.\]
	Hence $V(\Delta)=V(\lk(\sigma\backslash e_0))\cupdot V(\lk(\sigma\backslash e_1))\cupdot\cdots\cupdot V(\lk(\sigma\backslash e_{k-1}))=V(\lk(\sigma\backslash e_0))\cupdot V(\lk(e_0))$. Thus by Lemma \ref{lm: join of face links}, $\Delta$ is the join of the circles $\lk(\sigma\backslash e_{i})$, $0\leq i\leq k-1$. 
\end{proof}

The above two lemmas immediately imply the following theorem.
\begin{theorem}
	Let $\Delta$ be a flag homology $5$-manifold with $n$ vertices. Then $f_i(\Delta)=f_i(J_3(n))$ for some $1\leq i\leq 5$ if and only if $\Delta=J_3(n)$.
\end{theorem}
\begin{proof}
	The case $i=1$ is proved in \cite{Z}. Assume that $f_5(\Delta)=f_5(J_3(n))$. For inequality (\ref{eq: upper bound of f_1}) to be an equality, by Lemma \ref{lemma: sum of vertex links} we have $|\cup_{v\in \sigma\backslash\tau}V(\lk(v\cup\tau))|=f_0(\lk(\tau))$ for any 2-face $\tau$ and any facet $\sigma\supset \tau$. Hence by Lemma \ref{lemma: flag 2-sphere property}, the link of every 2-face in $\Delta$ is the suspension of circle. So by Lemma \ref{lemma: join of circles} and the proof of Proposition \ref{theorem: upper bound inequality}, it follows that $\Delta=C_1*C_2*C_3$, where $C_i$ are circles of length at least 4, and inequality (\ref{eq: even-dimension conjecture}) holds as an equality for every vertex link in $\Delta$. This means $C_i$ must have either $\left\lfloor \frac{n}{3}\right\rfloor$ or $\left\lceil \frac{n}{3}\right\rceil$ vertices, and $\Delta=J_3(n)$. The proof for $i=4$ is the same since $f_4(\Delta)=3f_5(\Delta)$. Also if $f_i(\Delta)=f_i(J_3(n))$ for $i=2$ or 3, then by Lemma \ref{lm: dehn-sommerville} and the fact that $f_j(\Delta)\leq f_j(J_3(n))$ for $j=1,5$, we must have that $f_1(\Delta)=f_1(J_3(n))$ and $f_5(\Delta)=f_5(J_3(n))$. Hence $\Delta=J_3(n)$.
\end{proof}

Recall that Gal conjectured that if $\Delta$ is a flag simplicial 3-sphere on $n$ vertices such that $f_1(\Delta)>\frac{1}{4}(n^2+2n+16)$, then $\Delta$ is the join of two circles, see Conjecture \ref{conj: Gal}. This conjecture is optimal; consider the flag 3-sphere $\Gamma$ obtained by taking the join of a $(\frac{n}{2}-1)$-cycle and a $\frac{n}{2}$-cycle and subdividing an edge whose link is the 4-cycle. Then $\Gamma$ is not the join of two circles, and $f_1(\Gamma)=\frac{1}{4}(n^2+2n+16)$. It was showed in \cite{AH1} that this conjecture holds asymptotically. We propose the following generalization of Gal's conjecture.
\begin{conjecture}\label{conj: almost join of cycles}
	For any $m\geq 2$, there exists a constant $b=b(m)$ such that if $\Delta$ be a flag homology $(2m-1)$-manifold with $n$ vertices and $f_{2m-1}(\Delta)> f_{2m-1}(J_m(n))-b(m)n^{m-1}$, then $\Delta$ is the join of $m$ cycles.
\end{conjecture}

In what follows we prove this conjecture for $m=2, 3$. Let $\Delta$ be a flag homology $(2m-1)$-manifold, $\sigma_0=\{v_1,v_2,\dots, v_{2m}\}$ a facet of $\Delta$, and $\sigma_i=\sigma\backslash\{v_i\}\cup\{u_i\}$ be the adjacent facets of $\sigma_0$. Define $m_{\sigma_i}=\sum_{v\in\sigma_i} f_0(\lk(v))$ for $0\leq i\leq 2m$ and $M_{\sigma_0}:=\sum_{i=0}^{2m} m_{\sigma_i}$. 

\begin{lemma}\label{lm: bound of M_sigma}
	Let $m=2$ or $3$ and let $\Delta$ be a flag normal $(2m-1)$-pseudomanifold with $n$ vertices. If $\Delta$ is not the join of $m$ cycles, then $M_{\sigma}< (1+2m) (2(m-1)n+4m)$ for any facet $\sigma\in\Delta$.
\end{lemma}

\begin{proof}
	First by Lemma \ref{lemma: sum of vertex links} we have that $M_\sigma\leq (1+2m) (2(m-1)n+4m)$. Assume by contradiction that there is a facet $\sigma_0=\{v_1, v_2,\cdots, v_{2m}\}$ such that $M_{\sigma_0}=(1+2m)(2(m-1)n+4m)$. In particular, we have $m_{\sigma_i}=2(m-1)n+4m$ for every $0\leq i\leq 2m$. By Lemma \ref{lemma: sum of vertex links}, for any $(2m-4)$-face $\tau\in \sigma_0$ and any facet $\sigma_i$ containing $\tau$, we have $\cup_{w\in \sigma_i\backslash \tau} V(\lk(w\cup\tau))=V(\lk(\tau))$. It follows from Lemma \ref{lemma: flag 2-sphere property} that $\lk(\tau)$ must be the suspension of a circle. Again by Lemma \ref{lemma: sum of vertex links}, $m_{\sigma_0}=2(m-1)n+4m$ implies that $\cup_{\delta\in\sigma_0, |\delta|=2m-2} V(\lk(\delta))=V(\Delta)$. Hence by Lemma \ref{lemma: join of circles}, $\Delta$ is the join of $m$ circles, which is a contradiction.
\end{proof}

We give the proof of Conjecture \ref{conj: almost join of cycles} for $m=2, 3$. 

\begin{proof}
	Assume that $\Delta$ is not the join of $m$ cycles. By Lemma \ref{lm: bound of M_sigma}, we have $M_\sigma< (1+2m)(2(m-1)n+4m)$ for any facet $\sigma\in\Delta$. Denote by $N(\sigma)$ the set of vertices in some adjacent facet of $\sigma$ but not in $\sigma$. It follows that
	\begin{align*}
	\begin{split}
	f_{2m-1}(\Delta)\left( 2(m-1)n+4m-\frac{1}{2m+1}\right)  &\geq \frac{1}{2m+1}\sum_{\sigma\in\Delta} M_{\sigma}\\
	&=\frac{1}{2m+1}\sum_{v\in\Delta}f_0(\lk(v))\cdot \Big(\#\{\sigma: v\in N(\sigma)\}+2m\cdot\#\{\sigma: v\in \sigma\}\Big)\\&=\sum_{v\in\Delta} f_0(\lk(v))f_{2m-2}(\lk(v)).
	\end{split}
	\end{align*}
	Since $\lk(v)$ is a flag homology sphere of dimension 2 or 4, by Gal's result \cite{G} the link $\lk(v)$ satisfies the inequality (\ref{eq: even-dimension conjecture}). The same argument as in the proof of Proposition \ref{theorem: upper bound inequality} yields that $$f_{2m-1}(\Delta)\left( 2(m-1)n+4m-\frac{1}{2m+1}\right) \geq 2n(m-1)\left( \frac{mf_{2m-1}(\Delta)}{n}\right) ^{\frac{m}{m-1}} +4mf_{2m-1}(\Delta),$$
	 \quad i.e. \quad 
	 \begin{equation} \label{eq}
	 \left( 1-\frac{1}{2n(m-1)(2m+1)}\right)^{m-1} \left( \frac{n}{m}\right) ^m \geq f_{2m-1}(\Delta).
	 \end{equation} 
	As in the proof of Proposition \ref{theorem: upper bound inequality}, $f_{2m-1}(\Delta)$ attains its maximum when 1) for all vertex $v\in \Delta$, $f_{2m-2}(\lk(v))=f_{2m-2}(J^*_{m-1}(f_0(\lk(v)))$, and 2) $f_0(\lk(v))$ among all vertex links as even as possible. Hence the leading term $(\frac{n}{m})^m$ in (\ref{eq}) can be corrected as $f_{2m-1}(J_{m}(n))$ and the result follows.
\end{proof}

 \section*{Acknowledgements}
I would like to thank the anonymous referees for valuable comments.
 \bibliographystyle{amsplain}
 
 \end{document}